\documentclass[12pt]{amsart}
\usepackage{graphicx}
\usepackage{amssymb}
\usepackage{amsmath}
\usepackage{amsthm}
\usepackage{amscd}
\usepackage{epsfig}
\allowdisplaybreaks[4]

\newtheorem{theorem}{Theorem}[section]

\newtheorem{lemma}[theorem]{Lemma}

\newtheorem{proposition}[theorem]{Proposition}

\newtheorem{remark}[theorem]{Remark}
\theoremstyle{definition}

\numberwithin{equation}{section}

\newcommand{\R}{\mathbb R}
\newcommand{\Z}{\mathbb Z}

\begin{document}

\title [Entropy dimension for Deterministic Walks in Random Sceneries]
{Entropy dimension for Deterministic Walks in Random Sceneries}

\author [Dou Dou and Kyewon Koh Park]{Dou Dou and Kyewon Koh Park}

\address{Department of Mathematics, Nanjing University,
Nanjing, Jiangsu, 210093, P.R. China} \email{doumath@163.com}

\address{Center for Mathematical Challenges, Korea Institute for Advanced Study, Seoul 130-722, Korea}
\email{kkpark@kias.re.kr}

\thanks{}


\begin{abstract}
Entropy dimension is an entropy-type quantity which takes values in $[0,1]$ and classifies different levels of intermediate growth rate of complexity for dynamical systems.
In this paper, we consider the complexity of skew products of irrational rotations with Bernoulli systems, which can be viewed as deterministic walks in random sceneries,
and show that this class of models can have any given entropy dimension by choosing suitable rotations for the base system.
\end{abstract}

\maketitle

\markboth{D. Dou and K.K. Park}{Entropy dimension for DWRS}

\section{Introduction}
Complexity for dynamical systems describes the growth rate of orbits under the iteration of the actions. We measure the complexity
via the iterated partitions in a probability measure space and via the iterated open covers in a topological setting. If a system has positive entropy,
then it has exponential growth rate. Positive entropy systems have been studied and much understood in recent decades and they are known to be chaotic and unpredictable.
Hence the word complexity is used mostly for entropy zero systems. Many examples of polynomial growth rate, such as irrational rotations, adding machines, interval exchange
maps and some billiards on polygons, are well-known. However the systems of
sub-exponential but super-polynomial growth rate (called intermediate growth rate) are not well understood yet, and hence are wide open to further development.

In the study of more general group actions, for example, $\Z^2-$systems, the sub-exponential growth arises
naturally and some of their properties have been investigated in \cite{Park1,Park2} and also in \cite{JLP2}.
The notions of slow entropy or entropy dimension have been introduced to distinguish different levels of intermediate growth rate of
zero entropy systems (\cite{KT,Ca,FP,H}). 
For $\Z$-actions, several definitions of entropy dimensions for topological and measurable dynamical systems are introduced and investigated in \cite{Ca,FP,DHP1,DHP2,ADP,DP}.

In \cite{C}, a class of uniformly recurrent infinite words with given intermediate complexity was constructed and it was shown in \cite{ADP} that the systems generated by the words are uniquely ergodic. Moreover it was shown in \cite{ADP} that
entropy dimension does not have the property of variational principle.
In \cite{DHP1,DHP2}, systems of uniform entropy dimension (u.d., for short) were constructed. They have the property that all finite open covers have the same topological entropy dimension which corresponds to topological property u.p.e. or all finite measurable partitions have
the same metric entropy dimension which corresponds to the metric property of K-mixing. Since the study of systems with intermediate complexity is still at the beginning stage, we have many interesting and basic questions. Besides the properties mentioned above, very little is known in the way of general results and the developments have been focused on examples.
Some examples related to intermediate complexity can be found in \cite{JLP, KP}, but all of theses examples are via constructive methods to have desired properties. Hence we would like to provide with examples of
intermediate growth rate out of more familiar models.

Aaronson showed the relative complexity of intermediate growth rate for a class of random walks in random sceneries (\cite{A}). However the model itself has positive entropy (exponential growth rate) since the base is Bernoulli. We consider a class of models which we call deterministic walks in random sceneries (DWRS, for short).
The deterministic walks are chosen to be the irrational rotations of the unit circle $\mathbf{T}$ and the random sceneries are chosen to be the orbits of Bernoulli systems (or more generally, positive entropy systems).
One can refer to \cite{AK,Av,CK} for DWRS's on $\mathbb{Z}$.

In our DWRS models, the unit circle of the base is divided into two subintervals of the same length. If after a rotation, a base point lies in the first interval, it will walk forward along its scenery and walk backward otherwise. Due to the construction, at time $n$, the point can only walk $o(n)$ distance away from the starting point along the orbit of the scenery system, hence the entropy should be zero. The scenery supplies the ``randomness'' for this skew-product system and the base controls the growth rate of the complexity.
By choosing suitable irrational rotation number $\alpha$ of the base, we will show that our DWRS model has intermediate growth rate. More precisely, for any $\tau\in[0,1]$, there exists $\alpha$ such that the skew-product system has the upper entropy dimension $\tau$.

We should notice that there are several ways to give the definitions of entropy dimensions. In this paper, we proceed with those appeared in \cite{ADP,DHP1,DHP2,DP} via entropy generating sequences. Due to the choice of taking $\limsup$ or $\liminf$ in the definitions of entropy dimensions, we have upper entropy dimensions and lower entropy dimensions. It is shown in \cite{DHP1,DHP2} that both of them are isomorphism invariants. However it is not hard to construct non-isomorphic systems of the same upper entropy dimension. In most of constructive examples we can increase upper dimension by adding more independencies and also decrease lower dimension by adding more repetitions. In fact making the upper and lower dimensions agree demands more careful choices of parameters at the successive steps (\cite{FP,DP,DHP2}). We remark that this is unlike computing the entropy of a system, which can never increase in longer blocks(names). When we consider factor maps between systems, the so called localization theory in dynamics or the u.d. property, the upper entropy dimension appears to carry more meaning (\cite{DHP1,DHP2}). Hence in this paper, when we talk about entropy dimension, it refers to the upper entropy dimension.

The paper is organized as follows. In section 2, we will give the definitions of entropy dimensions via entropy generating sequences in both
topological settings and measure-theoretic settings. Together with a brief review on
the best approximation and continued fraction expansion of irrationals, we will also introduce the construction and some preliminaries of the DWRS models. In section 3, we will discuss some recurrence properties of the irrational rotation and in section 4 and 5,
we will show the existence of any given topological and metric entropy dimensions by choosing the rotations properly.


\section{Preliminaries}

\subsection{Entropy dimension}

Let $S=\{s_1<s_2<\cdots\}$ be an increasing sequence of integers. We define the {\it upper dimension} and the {\it lower dimension} of the sequence $S$ by
$$\overline{D}(S)=\limsup_{n\rightarrow \infty}\frac{\log n}{\log s_n}$$ and $$\underline{D}(S)=\liminf_{n\rightarrow \infty}\frac{\log n}{\log s_n}$$
respectively. Equivalently,
\begin{align*}
  \overline{D}(S)=\inf \{\tau\ge 0:\limsup_{n\rightarrow \infty}\frac{n}{(s_n)^{\tau}}=0\}=\sup
\{\tau\ge 0:\limsup_{n\rightarrow \infty}\frac{n}{(s_n)^{\tau}}=\infty \}
\end{align*}
and
\begin{align*}
  \underline{D}(S)=\inf \{\tau\ge 0:\liminf_{n\rightarrow \infty}\frac{n}{(s_n)^{\tau}}=0\}=\sup
\{\tau\ge 0:\liminf_{n\rightarrow \infty}\frac{n}{(s_n)^{\tau}}=\infty \}.
\end{align*}
Clearly $0\le \underline{D}(S)\le \overline{D}(S)\le 1$. When
$\overline{D}(S)=\underline{D}(S)=\tau$, we say the sequence $S$ has
dimension $\tau$. For example, the dimension of the sequence $\{1,4,9,\cdots,n^2,\cdots\}$ is $\frac{1}{2}$ and
more general, the dimension of the sequence $\{n^t\}$ ($t$ is a positive integer) is $\frac{1}{t}$.

By a topological dynamical system (TDS, for short) we mean a pair $(X,T)$ where $X$ is a compact metric space and $T$ is a continuous onto self-map on $X$.
Let $(X,T)$ be a TDS and $\mathcal{U}$ be a finite open cover of $X$. We say an increasing sequence of positive integers $S=\{s_1<s_2<\cdots\}$ is an {\it entropy generating sequence} of $\mathcal{U}$ if $$\liminf_{n\rightarrow\infty}\frac{1}{n}\log N(\bigvee_{i=1}^nT^{-s_i}\mathcal{U})>0,$$
where $N(\mathcal V)$ denotes the minimal cardinality of subcovers of the open cover $\mathcal V$ of $X$.

Then the {\it (upper) entropy dimension} of $\mathcal{U}$ is defined by
$$\overline{D}(T, \mathcal{U})= \sup \{\overline{D}(S): S \text{ is an entropy generating sequence of }\mathcal{U}\}.$$
And the {\it (upper) topological entropy dimension} of the TDS $(X,T)$ is defined by
$$\overline{D}(X,T)=\sup_{\mathcal{U}} \overline{D}(T,\mathcal{U}),$$
where the supremum is taken over all the finite open covers $\mathcal{U}$ of $X$.

\begin{remark}\label{remark-1}
The following will be used for the computation of topological entropy dimension which appeared in \cite{DHP1}.
  \begin{enumerate}
    \item There are alternative equivalent definitions for $\overline{D}(T, \mathcal{U})$
    \begin{align*}
      \overline{D}(T, \mathcal{U})&= \inf \{\tau\ge 0:\limsup_{n\rightarrow \infty}\frac{1}{n^{\tau}}N(\bigvee_{i=0}^{n-1}T^{-i}\mathcal{U})=0\}\\
      &=\sup
\{\tau\ge 0:\limsup_{n\rightarrow \infty}\frac{1}{n^{\tau}}N(\bigvee_{i=0}^{n-1}T^{-i}\mathcal{U})=\infty \}.
    \end{align*}
    \item If $\mathcal{U}$ is a generating open cover, i.e. $\lim\limits_{n\rightarrow\infty}{\rm diam}(\bigvee_{i=0}^{n-1}T^{-i}\mathcal{U})=0$,
    then \begin{align*}
      \overline{D}(X,T)=\overline{D}(T,\mathcal{U}).
    \end{align*}
  \end{enumerate}
\end{remark}


Similarly, we can define entropy dimension for measurable dynamical systems. By a measurable dynamical system (MDS, for short), we mean a quadruple $(X,\mathcal{B}, \mu, T)$,
where $(X,\mathcal {B},\mu)$ is a probability space and $T$ is a measurable transformation on $X$ that preserves $\mu$.
Let $(X,\mathcal{B}, \mu, T)$ be a MDS and $\mathcal{P}$ be a finite measurable partition of $X$.
An increasing sequence of positive integers $S=\{s_1<s_2<\cdots\}$ is called an {\it entropy generating sequence} of $\mathcal{P}$ (w.r.t. $\mu$) if $$\liminf_{n\rightarrow\infty}\frac{1}{n}H_{\mu}(\bigvee_{i=1}^nT^{-s_i}\mathcal{P})>0.$$
Then the {\it (upper) entropy dimension} of $\mathcal{P}$ is defined by
$$\overline{D}_{\mu}(T, \mathcal{P})= \sup \{\overline{D}(S): S \text{ is an entropy generating sequence of }\mathcal{P}\text{ (w.r.t }\mu \text{) }\}.$$
The (upper) entropy dimension of the MDS $(X,\mathcal{B},\mu,T)$ is defined by
$$\overline{D}_{\mu}(X,T)=\sup_{\mathcal{P}} \overline{D}_{\mu}(T,\mathcal{P}),$$
where the supremum is taken over all the finite measurable partitions $\mathcal{P}$ of $X$. It is shown in \cite{DHP2}, $\overline{D}_{\mu}(X,T)$ equals
$\overline{D}_{\mu}(T,\mathcal{P})$ when $\mathcal{P}$ is a generating partition.

\subsection{DWRS}

Let $\mathbf{T}=\R/\Z$ be the unit circle and $\alpha\in (0,1)$ be an irrational number. Let $(\mathbf{T},\alpha)$ be the rotation by $\alpha$ on the unit circle. For any $\theta\in \mathbf{T}$, we
define a bi-infinite $\{-1,1\}$ sequence
$z=z(\theta)=(\cdots,z_{-1},z_0,z_1\cdots)$ by
$$z_n=\begin{cases} 1\, &\text{if } \theta+n\alpha\, \mod\, 1\in[0,\frac{1}{2})\\ -1 \, &\text{if }
\theta+n\alpha \, \mod\, 1\in[\frac{1}{2},1) \end{cases}.$$
Let $Z=cl\{\sigma^nz(0): n\in \Z\}$, where $\sigma$ is the left shift map. Then $(Z,\sigma)$ forms a TDS which is called
a Sturmian flow of type $(\alpha,\frac{1}{2})$ (see \cite[(2.53)]{V}). Since $\alpha$ and $\frac{1}{2}$ are rationally independent, $(Z,\sigma)$ is minimal and uniquely ergodic. Moreover, $(Z,\sigma)$ is an almost 1-1 extension of the irrational rotation $(\mathbf{T},\alpha)$. Let $\varphi$ be the factor map between $(Z,\sigma)$ and $(\mathbf{T},\alpha)$. It is known that $\varphi^{-1}(\theta)$ has exact two points if $\theta\in Orb(0)\cup Orb(\frac{1}{2})$ and $\varphi^{-1}(\theta)=z(\theta)$ otherwise.

Let $\mathcal{B}_{\mathbf{T}}$ and $\mathcal{B}_{Z}$ be the Borel $\sigma$-algebra
of $\mathbf{T}$ and $Z$ respectively. Let ${\bf m}$ and $\mu$ be the unique ergodic invariant probability measures of $(\mathbf{T},\alpha)$ and $(Z,\sigma)$ respectively.
Then ${\bf m}$ is the Haar-Lebesgue measure of $\mathbf{T}$ and moreover,
from the measure-theoretic point of view, the MDS's $(\mathbf{T}, \mathcal{B}_{\mathbf{T}}, {\bf m},\alpha)$ and
$(Z, \mathcal{B}_{Z}, \mu, \sigma)$ are conjugate.

Sometimes we will not distinguish $\mathbf{T}$ and $Z$ in this paper.

Let $(Y,S)$ be an invertible TDS. We define the map $T: Z\times Y\rightarrow Z\times Y$
by
$$T(z,y)=(\sigma z,S^{z_0}y),$$ for $z=(\cdots,z_{-1}z_0z_1\cdots)\in Z$. That is, for any point $(z,y)\in Z\times
Y$, if $z_0$ is $1$, we then push $y$ forward to $Sy$, otherwise we
pull $y$ backward to $S^{-1}y$. Hence we have constructed a skew-product system $(Z\times Y, T)$.
For measure-theoretic case, for an invertible MDS $(Y,\mathcal {D},\nu,S)$,
this skew-product system can also be viewed as a deterministic walk along the orbits
of $(Y,\mathcal {D},\nu,S)$.

For simplicity, we let $\mathfrak{Y}$ denote either the invertible TDS $(Y,S)$ or the invertible MDS $(Y,\mathcal {D},\nu,S)$.
And denote by $\mathfrak{X}=\mathfrak{X}_{\alpha,\mathfrak{Y}}$ either the TDS $(Z\times Y, T)$ or the MDS $(Z\times Y,\mathcal{B}_Z\otimes \mathcal{D},\mu\times \nu, T)$ and call it a deterministic walk with rotation $\alpha$ along $\mathfrak{Y}$.
Especially, when $\mathfrak{Y}$ is a Bernoulli system (or a positive entropy system), we call $\mathfrak{X}$ a DWRS.

\begin{remark}\label{conjugate}
  Since $(\mathbf{T}, \mathcal{B}_{\mathbf{T}}, {\bf m},\alpha)$ and $(Z, \mathcal{B}_{Z}, \mu, \sigma)$ are conjugate, sometimes $\mathfrak{X}$ is also viewed as
$(\mathbf{T}\times Y,\mathcal{B}_{\mathbf{T}}\otimes \mathcal{D},{\bf m}\times \nu, T)$. In this case, $T$ also denotes the corresponding skew-product map on $\mathbf{T}\times Y$.
\end{remark}
It is easy to see that the following hold and we omit the proofs.

\begin{proposition}
$\mathfrak{X}$ is minimal (respectively, transitive or ergodic) if and only if $\mathfrak{Y}$ is.
\end{proposition}

\begin{proposition}
For every $\alpha\in\mathbf{T}$ and any system $\mathfrak{Y}$, the entropy of $\mathfrak{X}$ is zero.
\end{proposition}

Let $z=(\cdots,z_{-1}z_0z_1\cdots)\in Z$. For any $m,n\in \mathbb{N}$ with $m\le n$, denote by
$$\omega(z,m,n)=\sum_{i=m}^nz_i, \; \omega(z,1,n)=\omega(z,n).$$
For $\theta\in \mathbf{T}$, denote by
$$\omega (\theta,m,n)=\omega (z(\theta),m,n), \ \ \omega (\theta,n)=\omega (z(\theta),n)=\omega(\theta,1,n)$$
and
$$\omega(n)=\omega(0,n)=\omega (z(0),n).$$

If $\varphi(z)$ does not belong to the orbits of $0$ and $\frac{1}{2}$, then $\omega(z,k)=\omega(\varphi(z),k)$.
Otherwise, $|\omega(z,k)-\omega(\varphi(z),k)|\le 2$.

Denote by
$$M(\theta,n)=\max_{1\le\ell\le n} \omega(\theta,\ell),\ \ m(\theta,n)=\min_{1\le\ell\le n}\omega(\theta,\ell)$$ and
$$M(z,n)=\max_{1\le\ell\le n}\omega(z,\ell),\ \ m(z,n)=\min_{1\le\ell\le n} \omega(z,\ell).$$

If $\varphi(z)$ does not belong to the orbits of $0$ and $\frac{1}{2}$, then $M(z,n)=M(\varphi(z),n)$ and $m(z,n)=m(\varphi(z),n)$.
Otherwise, $|M(z,n)-M(\varphi(z),n)|\le 2$ and $|m(z,n)-m(\varphi(z),n)|\le 2$.

Let $\theta\in\mathbf{T}$ and $z=z(\theta)=(\cdots,z_{-1}z_0z_1\cdots)\in Z$.
Recall that $T(z,y)=(\sigma z,S^{z_0}y)$. Hence, for $i\ge 0$,
\begin{align}\label{6}
T^i(z,y)&=(\sigma^i z, S^{\omega(z,0,i-1)}y),\nonumber \\
T^{-i}(z,y)&=(\sigma^{-i} z, S^{-\omega(\sigma ^{-i}z,0,i-1)}y),
\end{align}
where $\omega(z,0,-1):=0$.
\subsection{The best approximation and continued fraction expansion of irrationals}
Fix an irrational number $\alpha\in (0,1)$. The growth of $\omega (n)$
depends on the recurrent properties of $\alpha$. Here we collect some basic facts on the continued fraction expansion of
irrationals (see Khinchin's classic book \cite{Kh} for details).

For a real number $\beta$, let $\|\beta\|=\min\{\{\beta\},1-\{\beta\}\}$, where $\{\beta\}$ denotes the
fractional part of $\beta$.

Let $[a_0;a_1,a_2,\cdots]$ be the continued fraction expansion of $\alpha$ and $\frac{p_n}{q_n}=[a_0;a_1,a_2,\cdots,a_n]$ be the continued fraction approximation of $\alpha$. Then
$$\frac{p_0}{q_0}<\frac{p_2}{q_2}<\cdots<\alpha<\cdots<\frac{p_3}{q_3}<\frac{p_1}{q_1}$$
and
\begin{equation}\label{4}
    q_n=a_nq_{n-1}+q_{n-2},
\end{equation}
where $q_{-1}=0$ and $q_{-2}=1$. Since $\alpha$ is chosen from $(0,1)$, we have that $a_0=0$ and $q_0=1$.
Moreover, $p_n,q_n$'s have the following properties:
\begin{equation}\label{1}
    \|q_n\alpha\|=\begin{cases} \{q_n\alpha\}\, &\text{if } n=2k, \\
1-\{q_n\alpha\} \, &\text{if } n=2k+1;
\end{cases}
\end{equation}

\begin{equation}\label{2}
    \|q_{n-1}\alpha\|<\|k\alpha\|,\forall k<q_n,k\neq q_{n-1};
\end{equation}

\begin{equation}\label{3}
    \frac{1}{q_n(q_n+q_{n+1})}\le (-1)^n(\alpha-\frac{p_n}{q_n})\le\frac{1}{q_nq_{n+1}}.
\end{equation}

\section{Some recurrence properties of $(\mathbf{T},\alpha)$}

Let $\alpha$ be a fixed irrational in $[0,1)$. In this section we will discuss the growth of $\omega (\theta,n)$ for $\theta\in \mathbf{T}$, which will be related
with the complexity of our DWRS models. Firstly, let's study the growth of $\omega(n)$. For convenience, we will identify $\mathbf{T}$ with the interval $[0,1)$.
When we talk about some point $\ell\alpha$, it should be understood as $\ell\alpha \mod 1$ due to the context. And when we talk about some interval, it should be understood
as an arc of $\mathbf{T}$.
For each $n\in\mathbb{N}$, let $Q_n=\{k\alpha: 1\le k\le q_n\}\subset\mathbf{T}$.

The following two propositions will show that the points in $Q_n$ have ``almost equi-distribution" properties on $\mathbf{T}$ with respect to any translation of the partition
$\{[0,\frac{1}{2}),$ $[\frac{1}{2},1)\}$.
\begin{proposition}\label{lemma-qn}
$$\omega(q_n)=\begin{cases}0\, &\text{ if } q_n=2r,\\ 1\,&\text{ if } q_n=2r+1 \text{ and } n=2k,\\ -1\, &\text{ if }q_n=2r+1 \text{ and }n=2k+1.
\end{cases}$$
\end{proposition}
\begin{proof}
We will just prove the case $n=2k$ since the parallel argument holds for $n=2k+1$.

We note that $\|q_n\alpha\|=\{q_n\alpha\}$.

For $1\le\ell\le q_n-1$,
if $\ell\alpha\in (0,\frac{1}{2})$, then $-\ell\alpha\in
(\frac{1}{2},1)$. Hence $(q_n-\ell)\alpha\in
(\frac{1}{2}+\|q_n\alpha\|,1+\|q_n\alpha\|)$. Since no element in $Q_n$ belongs to $[0,\|q_n\alpha\|)$, we know that $(q_n-\ell)\alpha\in
(\frac{1}{2}+\|q_n\alpha\|,1)$. Likewise, if $\ell\alpha\in
(\frac{1}{2}+\|q_n\alpha\|,1)$, then
$(q_n-\ell)\alpha\in(0,\frac{1}{2})$. Moreover,
$\ell\alpha\in(\frac{1}{2},\frac{1}{2}+\|q_n\alpha\|)$ if and only
if $(q_n-\ell)\alpha\in(\frac{1}{2},\frac{1}{2}+\|q_n\alpha\|)$.

If $q_n=2r$, then $r\alpha=\frac{1}{2}+\frac{1}{2}||q_n\alpha||\in
(\frac{1}{2},\frac{1}{2}+\|q_n\alpha\|)$.  Note that $q_n\alpha\in [0,\frac{1}{2})$, hence $\omega(q_n)=0$.
If $q_n=2r+1$, then no element of the set
$Q_n$ belongs to
$(\frac{1}{2},\frac{1}{2}+\|q_n\alpha\|)$ since otherwise there will be two elements of $Q_n$ belonging to this interval, which contradicts to the fact that
$\frac{p_n}{q_n}$ is the best approximation of $\alpha$. Hence $\omega(q_n)=1$.
\end{proof}

\begin{proposition}\label{lemma-kq}
For $\theta\in \mathbf{T}$, the following hold.
\begin{enumerate}
\item If $\omega(q_n)=0$, then $\omega(\theta,q_n)=0\text{ or }\pm2$.
\item If $\omega(q_n)=\pm 1$, then $\omega(\theta,q_n)=\pm1\text{ or }\pm 3$.
\item If $q_n=2r+1$ and $kq_n<q_{n+1}$, then $|\omega(kq_n)|\le k$.
\end{enumerate}
\end{proposition}
\begin{proof}
(1). Suppose $\omega(q_n)=0$. If $\theta=0$ or $\frac{1}{2}$, obviously $\omega(\theta,q_n)=0$. Now we treat the case for $0<\theta<\frac{1}{2}$ and the case for $\frac{1}{2}<\theta<1$ is similar. Then
\begin{align*}
  \omega(\theta,q_n)&=\#\big ((\theta+ Q_n)\cap [0,\frac{1}{2})\big)-\#\big((\theta+ Q_n)\cap [\frac{1}{2},1)\big)\\
  &=\#\bigg(Q_n\cap \big([-\theta,0)\cup [0,\frac{1}{2}-\theta)\big)\bigg)-\#\bigg(Q_n\cap \big([\frac{1}{2}-\theta,\frac{1}{2})\cup [\frac{1}{2},-\theta)\big)\bigg)
  \\
  &=\#\big(Q_n\cap [-\theta,0)\big)+\#\big(Q_n\cap [0,\frac{1}{2})\big)-\#\big(Q_n\cap [\frac{1}{2}-\theta,\frac{1}{2})\big)\\
  &\ \ \ \ -\#\big(Q_n\cap [\frac{1}{2}-\theta,\frac{1}{2})\big)-\#\big(Q_n\cap [\frac{1}{2},1)\big)+\#\big(Q_n\cap [-\theta,0)\big).
\end{align*}
Since $\omega(q_n)=\#\big(Q_n\cap [0,\frac{1}{2})\big)-\#\big(Q_n\cap [\frac{1}{2},1)\big)=0$, we have that
\begin{align}\label{eq-3-2-1}
  \omega(\theta,q_n)=2\bigg(\#\big(Q_n\cap [-\theta,0)\big)-\#\big(Q_n\cap [\frac{1}{2}-\theta,\frac{1}{2})\big)\bigg).
\end{align}
From Proposition \ref{lemma-qn}, $\omega(q_n)=0$ if and only if $q_n=2r$ for some $r$. As before we will just discuss for the case $n=2k$.
Note that in this case, $||q_n\alpha||=\{q_n\alpha\}$ and $r\alpha=\frac{1}{2}+\frac{1}{2}||{q_n\alpha}||$. Hence
$Q_n\cap[\frac{1}{2}-\frac{1}{2}||{q_n\alpha}||,\frac{1}{2}+\frac{1}{2}||{q_n\alpha}||)=\emptyset$.

For $0< \ell\le r$,  we have that
\begin{align}\label{eq-3-2-2}
  \ell\alpha\in [-\theta,0)&\Leftrightarrow (\ell+r)\alpha\in[\frac{1}{2}-\theta+\frac{1}{2}||{q_n\alpha}||, \frac{1}{2}+\frac{1}{2}||{q_n\alpha}||)\\
  &\Leftrightarrow (\ell+r)\alpha\in[\frac{1}{2}-\theta+\frac{1}{2}||{q_n\alpha}||, \frac{1}{2})\nonumber
\end{align}
For $r<\ell\le q_n$,  we have that
\begin{align}\label{eq-3-2-3}
  \ell\alpha\in [-\theta,0)&\Leftrightarrow (\ell-r)\alpha\in[\frac{1}{2}-\theta-\frac{1}{2}||{q_n\alpha}||, \frac{1}{2}-\frac{1}{2}||{q_n\alpha}||)\\
  &\Leftrightarrow (\ell-r)\alpha\in[\frac{1}{2}-\theta-\frac{1}{2}||{q_n\alpha}||, \frac{1}{2})\nonumber
\end{align}
Notice that $[\frac{1}{2}-\theta-\frac{1}{2}||{q_n\alpha}||,\frac{1}{2}-\theta+\frac{1}{2}||{q_n\alpha}||)$ contains at most one point of $Q_n$ (since the length of this interval is $||{q_n\alpha}||$).
Together with \eqref{eq-3-2-1}, \eqref{eq-3-2-2} and \eqref{eq-3-2-3}, we conclude that
\begin{align*}
   \omega(\theta,q_n)=
   \begin{cases}
     2, \text{ if }(\ell-r)\alpha\in [\frac{1}{2}-\theta-\frac{1}{2}||{q_n\alpha}||, \frac{1}{2}-\theta) \text{ for some } r<\ell\le q_n;\\
     -2, \text{ if }(\ell+r)\alpha\in [\frac{1}{2}-\theta,\frac{1}{2}-\theta+\frac{1}{2}||{q_n\alpha}||) \text{ for some } 0<\ell\le r;\\
     0, \text{ otherwise}.
   \end{cases}
\end{align*}

(2). We just consider the case for $\omega(q_n)=1$ and the case for $\omega(q_n)=-1$ is similar.

By Proposition \ref{lemma-qn}, if $\omega(q_n)=1$ then $n=2k$ and $q_n=2r+1$.

It is clear that $\omega(\theta,q_n)=1$ if $\theta=0$ and $\omega(\theta,q_n)=-1$ if $\theta=\frac{1}{2}$. Now assume $0<\theta<\frac{1}{2}$.
Similar with \eqref{eq-3-2-1},
we have
\begin{align}\label{eq-3-2-4}
  \omega(\theta,q_n)=2\bigg(\#\big(Q_n\cap [-\theta,0)\big)-\#\big(Q_n\cap [\frac{1}{2}-\theta,\frac{1}{2})\big)\bigg)+1.
\end{align}

Assume that $\frac{1}{2}\in (s_1\alpha,s_2\alpha)$, where $s_1\alpha,s_2\alpha\in Q_n$ and $(s_1\alpha,s_2\alpha)\cap Q_n=\emptyset$.
If $s_2\le q_n-q_{n-1}$, then $(s_2+q_{n-1})\alpha\in Q_n$ and no element of $Q_n$ can lie between $(s_2+q_{n-1})\alpha$ and $s_2\alpha$.
Hence $s_1=s_2+q_{n-1}$ and $|s_2\alpha-s_1\alpha|=||q_{n-1}\alpha||$.
If $s_2> q_n-q_{n-1}$, then $(s_2+q_{n-1}-q_n)\alpha\in Q_n$ and no element of $Q_n$ can lie between $(s_2+q_{n-1}-q_n)\alpha$  and $s_2\alpha$.
Hence $s_1=s_2+q_{n-1}-q_n$ and $|s_2\alpha-s_1\alpha|=||q_{n-1}\alpha||+||q_{n}\alpha||$. So $(-\theta+s_1\alpha,-\theta+s_2\alpha)$ contains at most two elements in $Q_n$.

For $0<\ell\le q_n-s_2$, since $Q_n\cap[\frac{1}{2},s_2\alpha)=\emptyset$, we have that
\begin{align}\label{eq-3-2-5}
  \ell\alpha\in [-\theta,0)&\Leftrightarrow (\ell+s_2)\alpha\in[-\theta+s_2\alpha, s_2\alpha)\\
  &\Leftrightarrow (\ell+s_2)\alpha\in[-\theta+s_2\alpha, \frac{1}{2}).\nonumber
\end{align}
For $q_n-s_2<\ell\le q_n$,  since no element of $Q_n$ lies between $s_2\alpha-||{q_n\alpha}||$ and $\frac{1}{2}$, we have that
\begin{align}\label{eq-3-2-6}
  \ell\alpha\in [-\theta,0)&\Leftrightarrow (\ell+s_2-q_n)\alpha\in[-\theta+s_2\alpha-||{q_n\alpha}||, s_2\alpha-||{q_n\alpha}||)\\
  &\Leftrightarrow (\ell+s_2-q_n)\alpha\in[-\theta+s_2\alpha-||{q_n\alpha}||, \frac{1}{2}).\nonumber
\end{align}

In the proof of Proposition \ref{lemma-qn} (the last paragraph), we have shown that $Q_n$ does not intersect with the interval $(\frac{1}{2}, \frac{1}{2}+||q_n\alpha||)$ when $q_n=2r+1$. So
$s_2\alpha-||q_n\alpha||\ge \frac{1}{2}$. Noticing that $[-\theta+\frac{1}{2},-\theta+s_2\alpha)$ contains at most two elements of $Q_n$ and
$[-\theta+\frac{1}{2},-\theta+s_2\alpha-||q_n\alpha||)$ contains at most one element of $Q_n$,
together with \eqref{eq-3-2-4}, \eqref{eq-3-2-5} and \eqref{eq-3-2-6}, we conclude that $\omega(\theta,q_n)=\pm 1\text{ or }-3$.

If $\frac{1}{2}< \theta<1$, then $\omega(\theta,q_n)=-\omega(\theta-\frac{1}{2},q_n)$, and hence $\omega(\theta,q_n)=\pm 1\text{ or }3$.

(3). If $q_n=2r+1$ and $n=2t$, i.e. $\omega(q_n)=1$, then $||q_n\alpha||=\{q_n\alpha\}$. We now order the points in $Q_n$ on $\mathbf{T}$ by
$0<q_n\alpha<\cdots<s_1\alpha<\frac{1}{2}<s_2\alpha<\cdots<q_{n-1}\alpha<1$. Noticing that $kq_n<q_{n+1}$,  hence
the points in $\{\ell\alpha\}_{\ell=1}^{kq_n}$ should be ordered as
\begin{align*}
  0&<q_n\alpha<2q_n\alpha<\cdots<kq_n\alpha\\
  &\qquad \qquad \cdots\\
  &<s_1\alpha<(s_1+q_n)\alpha<\cdots<\big(s_1+(k-1)q_n\big)\alpha\\
  &<s_2\alpha<(s_2+q_n)\alpha<\cdots<\big(s_2+(k-1)q_n\big)\alpha\\
  &\qquad \qquad \cdots\\
  &<q_{n-1}\alpha<(q_{n-1}+q_n)\alpha<\cdots<\big(q_{n-1}+(k-1)q_n\big)\alpha<1.
\end{align*}
Due to the place where $\frac{1}{2}$ lies, obviously $|\omega(kq_n)|\le k$.

Similarly, when $q_n=2r+1$ and $n=2t+1$, we also have $|\omega(kq_n)|\le k$.
\end{proof}

The following proposition gives estimation on $m(\theta,q_n)$ and $M(\theta,q_n)$.

\begin{proposition}\label{lemma-Mm}
Let $\theta\in \mathbf{T}$, then we have the following estimation.
\begin{enumerate}
\item
For $q_n<t\le q_{n+1}$, $M(\theta,t)\le M(\theta,q_n)+3\lceil\frac{t}{q_n}\rceil$.
\item
For $q_n<t\le q_{n+1}$, $m(\theta,t)\ge m(\theta,q_n)-3\lceil\frac{t}{q_n}\rceil$.
\item
$\max\{M(0,q_{n+1}),|m(0,q_{n+1})|\}\ge
|\omega(q_n)|\lfloor\frac{q_{n+1}}{6q_n}\rfloor.$
\end{enumerate}
\end{proposition}
\begin{proof}
For any $0<\ell\le t$, we write $\ell=\ell_0+kq_n$, where $0\le\ell_0<q_n$ and $0\le k\le \lfloor\frac{t}{q_n}\rfloor$.
Then
\begin{align*}
\omega(\theta,\ell)&=\omega(\theta,\ell_0)+\omega(\theta,\ell_0+1,\ell_0+q_n)+\cdots+\omega(\theta,\ell_0+1+(k-1)q_n,\ell_0+kq_n) \\
&=\omega(\theta,\ell_0)+\omega(\theta+\ell_0\alpha,q_n)+\cdots+\omega(\theta+(\ell_0+(k-1)q_n)\alpha,q_n).
\end{align*}
Hence by Proposition \ref{lemma-kq},
$$m(\theta,q_n)-3k\le \omega(\theta,\ell)\le M(\theta,q_n)+3k.$$
This proves (1) and (2).

Now we prove (3). Obviously (3) is true if $\omega(q_n)=0$. Without
loss of generality, we assume that $\omega(q_n)=1$.

Suppose $s_1\alpha<\frac{1}{2}<s_2\alpha$, where $0<s_1,s_2<q_n$ and no element in $Q_n$ lies between $s_1\alpha$ and $s_2\alpha$.
From the property of $\frac{p_n}{q_n}$ (inequality \eqref{2}), $|s_1-s_2|\le q_{n-1}$ and $\|s_1\alpha-s_2\alpha\|\ge\|q_{n-1}\alpha\|$.
If $\frac{1}{2}-\{s_1\alpha\}<||q_n\alpha||$, then
\begin{equation}\label{relation-1}
    \omega(0,(k-1)q_n+1,kq_n)=
    \begin{cases}1, &\text{ if } k=1,\\
                 -1,& \text{ if } k=2, 3, \cdots, \lfloor\frac{\|q_{n-1}\alpha\|}{\|q_n\alpha\|}\rfloor.
    \end{cases}
\end{equation}
If $\frac{1}{2}-\{s_1\alpha\}\ge||q_n\alpha||$, then
\begin{equation}\label{relation-1'}
    \omega(0,(k-1)q_n+1,kq_n)=
    \begin{cases}1, &\text{ if } k=1,2,\cdots, \lfloor\frac{\frac{1}{2}-\{s_1\alpha\}}{\|q_n\alpha\|}\rfloor,\\
                 -1,& \text{ if } k=\lfloor\frac{\frac{1}{2}-\{s_1\alpha\}}{\|q_n\alpha\|}\rfloor+1, \cdots, \lfloor\frac{\|q_{n-1}\alpha\|}{\|q_n\alpha\|}\rfloor.
    \end{cases}
\end{equation}

 We also note that by inequality \eqref{3},
\begin{align*}
    \frac{\|q_{n-1}\alpha\|}{\|q_n\alpha\|}&\ge  \frac{1}{q_{n-1}(q_{n-1}+q_n)}q_nq_{n+1}\ge \lfloor\frac{q_{n+1}}{q_n}\rfloor\frac{q_n^2}{q_{n-1}(q_{n-1}+q_n)}\nonumber\\
    &\ge\frac{1}{2}\lfloor\frac{q_{n+1}}{q_n}\rfloor\ge\lfloor\frac{q_{n+1}}{2q_n}\rfloor.
\end{align*}
This implies that
\begin{equation}\label{relation-2}
\lfloor\frac{q_{n+1}}{6q_n}\rfloor \|q_n\alpha\|\le \frac{1}{3}\|q_{n-1}\alpha\|.
\end{equation}

We assume that $q_{n+1}\ge 6q_n$ (otherwise (3) already holds).

If $\frac{1}{2}-s_1\alpha>\frac{1}{3}\|q_{n-1}\alpha\|(\ge ||q_n\alpha||)$, then $\lfloor\frac{\frac{1}{2}-\{s_1\alpha\}}{||q_n\alpha||}\rfloor\ge \lfloor\frac{q_{n+1}}{6q_n}\rfloor$. By \eqref{relation-1'}, $\omega(kq_n)=k$ for any $k=1,2,\cdots, \lfloor\frac{\frac{1}{2}-\{s_1\alpha\}}{\|q_n\alpha\|}\rfloor$. So in this case $M(0,q_{n+1})\ge|\omega(q_n)|\lfloor\frac{q_{n+1}}{6q_n}\rfloor$.

If $||q_n\alpha||\le \frac{1}{2}-s_1\alpha\le\frac{1}{3}\|q_{n-1}\alpha\|$, then by \eqref{relation-1'}, $\omega(kq_n)=2\lfloor\frac{\frac{1}{2}-\{s_1\alpha\}}{\|q_n\alpha\|}\rfloor-k$ for any
$\lfloor\frac{\frac{1}{2}-\{s_1\alpha\}}{\|q_n\alpha\|}\rfloor<k\le \lfloor\frac{\|q_{n-1}\alpha\|}{\|q_n\alpha\|}\rfloor$.
So in this case
\begin{align*}
    m(0,q_{n+1})&\le \omega(\lfloor\frac{\|q_{n-1}\alpha\|}{\|q_n\alpha\|}\rfloor q_n)=2\lfloor\frac{\frac{1}{2}-\{s_1\alpha\}}{\|q_n\alpha\|}\rfloor-\lfloor\frac{\|q_{n-1}\alpha\|}{\|q_n\alpha\|}\rfloor\\
                &\le-\lfloor\frac{q_{n+1}}{6q_n}\rfloor.
\end{align*}

If $\frac{1}{2}-s_1\alpha\le\frac{1}{3}\|q_{n-1}\alpha\|$ and $\frac{1}{2}-\{s_1\alpha\}<||q_n\alpha||$, then by \eqref{relation-1}, it obviously holds that
\begin{align*}
    m(0,q_{n+1})&\le \omega(\lfloor\frac{\|q_{n-1}\alpha\|}{\|q_n\alpha\|}\rfloor q_n)=-\lfloor\frac{\|q_{n-1}\alpha\|}{\|q_n\alpha\|}\rfloor+2\\
                &\le-\lfloor\frac{q_{n+1}}{6q_n}\rfloor.
\end{align*}
\end{proof}

Denote
\begin{equation*}
  M(k)=\max_{z\in Z}M(z,k) \text{ and }m(k)=\min_{z\in Z}m(z,k).
\end{equation*}
\begin{proposition}\label{lemma-Mm2}
  For $q_n< k\le q_{n+1}$, we have
$$M(k)\le q_1+3\lceil\frac{k}{q_{n}}\rceil+3\sum_{j=2}^{n}\lceil\frac{q_j}{q_{j-1}}\rceil+2$$
and
$$m(k)\ge -q_1-3\lceil\frac{k}{q_{n}}\rceil-3\sum_{j=2}^n\lceil\frac{q_j}{q_{j-1}}\rceil-2.$$
\end{proposition}
\begin{proof}
  By (1) of Proposition \ref{lemma-Mm}, for $q_n< k\le q_{n+1}$,
  \begin{align*}
    M(\theta,k)&\le M(\theta,q_n)+3\lceil\frac{k}{q_n}\rceil\le \cdots\\
    &\le M(\theta,q_1)+3\lceil\frac{k}{q_{n}}\rceil+3\sum_{j=2}^n\lceil\frac{q_{j}}{q_{j-1}}\rceil\\
    &\le q_1+3\lceil\frac{k}{q_{n}}\rceil+3\sum_{j=2}^n\lceil\frac{q_{j}}{q_{j-1}}\rceil.
  \end{align*}
  Hence
  \begin{align*}
    M(k)&=\max_{z\in Z}M(z,k)\le \max_{\theta\in \mathbf{T}}M(\theta,k)+2\\
    &\le q_1+3\lceil\frac{k}{q_{n}}\rceil+3\sum_{j=2}^n\lceil\frac{q_{j}}{q_{j-1}}\rceil+2.
  \end{align*}
The proof is similar for $m(k)$.
\end{proof}

\section{Topological entropy dimension}
Let $\mathfrak{X}=\mathfrak{X}_{\alpha,\mathfrak{Y}}$ be the DWRS with rotation $\alpha$ along $\mathfrak{Y}$, where $\mathfrak{Y}$ is an invertible TDS with positive entropy. In this section we will show that for any $\tau \in [0,1)$, there exists $\alpha$ such that $\mathfrak{X}$ has topological entropy dimension $\tau$.

Let $\tau \in [0,1)$ be given. For any positive integers $\{t_1,t_2,\cdots, t_k\}$, we define an irrational number $\alpha(t_1,t_2,\cdots,t_k)=[0;a_1,a_2,\cdots]\in [0,1)$ as follows:
\begin{align}\label{def-alpha}
  &a_n=t_n\text{ for }1\le n\le k;
  \begin{cases}
    a_{k+1}=a_{k+2}=2, \quad\text{ if }q_kq_{k-1} \text{ is odd},\\
    a_{k+1}=3, a_{k+2}=2,\text{ if }q_{k-1}\text{ is even and }q_k\text{ is odd},\\
    a_{k+1}=a_{k+2}=3,\quad\text{ if }q_{k}\text{ is even and }q_{k-1}\text{ is odd};
  \end{cases}\\
  &a_{n+1}=2\lfloor q_n^{\frac{\tau}{1-\tau}}\rfloor \text{ for }n>k+2.\nonumber
\end{align}
By \eqref{4}, $q_{n+1}=a_{n+1}q_{n}+q_{n-1}$ is odd for every $n\ge k$. We remark here that if $q_n$ is even then $q_{n-1}$ must be odd (otherwise $q_{n-2},\cdots, q_1, q_0$ are all even, which contradicts to the fact $q_0=1$).

Denote by $E_{\tau}$ the collection of such irrational number $\alpha(t_1,t_2,\cdots,t_k)$'s. Then $E_{\tau}$ is dense in $[0,1)$.

Let $\alpha\in E_{\tau}$ be fixed. By \eqref{4}, a simple estimation shows that for any $c>0$,
\begin{equation}\label{eq-4-1}
  \frac{q_{n+1}}{q_n}\sim a_{n+1}\sim 2q_n^{\frac{\tau}{1-\tau}}\gg n^c,
\end{equation}
when $n$ is sufficiently large.

Let $l_n=\lfloor\frac{q_{n+1}}{12q_n}\rfloor$. By \eqref{3},
\begin{align*}
  \frac{||q_{n-1}\alpha||}{||q_n\alpha||}\ge \frac{\frac{1}{q_{n-1}+q_n}}{\frac{1}{q_{n+1}}}\ge \frac{q_{n+1}}{2q_n}.
\end{align*}
Hence
\begin{align}\label{eq-4-3}
 l_n\le \lfloor\frac{||q_{n-1}\alpha||}{6||q_n\alpha||}\rfloor.
\end{align}
Moreover, by \eqref{eq-4-1}, there exists $n_0$ such that $l_n>0$ and $\sum_{i=n_0}^{n-1}l_i<l_n$ whenever $n\ge n_0$.

Now we set
\begin{equation}\label{eq-4-2}
  F=\bigcup_{n=n_0}^\infty\{q_n,2q_n,\cdots, l_nq_n\}=\{s_1<s_2<\cdots\}.
\end{equation}

\begin{lemma}\label{lemma-4-1}
$$\overline{D}(F)\ge \tau.$$
\end{lemma}
\begin{proof}
From the definition of $F$,
\begin{align*}
  \overline{D}(F)&=\limsup_{n\rightarrow\infty}\frac{\log n}{\log s_n}\ge \lim_{m\rightarrow\infty}\frac{\log \sum_{i=n_0}^{m}l_i}{\log (l_mq_m)}\\
  &\ge \lim_{m\rightarrow\infty}\frac{\log l_m}{\log (l_mq_m)}=\lim_{m\rightarrow\infty}\frac{\log \lfloor\frac{q_{m+1}}{12q_m}\rfloor}{\log (\lfloor\frac{q_{m+1}}{12q_m}\rfloor q_m)}\\
  &=\lim_{m\rightarrow\infty}\frac{\log \frac{q_{m+1}}{q_m}}{\log q_{m+1}}=\frac{\frac{\tau}{1-\tau}}{\frac{\tau}{1-\tau}+1}\text{ (by \eqref{eq-4-1})}\\
  &=\tau.
\end{align*}
\end{proof}

The following lemma will be used while computing the entropy dimension.
\begin{lemma}\label{lem-recurrence}
When $n$ is large enough,
$${\bf m}\{\theta\in \mathbf{T}: \omega(\theta, iq_n)=i,
1\le i \le \lfloor\frac{||q_{n-1}\alpha||}{6||q_n\alpha||}\rfloor\}>\frac{1}{8}.$$
\end{lemma}
\begin{proof}
We only consider the case $n=2t$ and the case $n=2t+1$ is similar. By Proposition \ref{lemma-qn},
$\omega(q_n)=1$. And we note that $\{q_n\alpha\}=||q_n\alpha||$.

Order the points in $Q_n$ on $\mathbf{T}$ by
$0<q_n\alpha<\cdots<s_1\alpha<\frac{1}{2}<s_2\alpha<\cdots<q_{n-1}\alpha<1$. Then $||s_2\alpha-s_1\alpha||\ge ||q_{n-1}\alpha||$.
We note that from \eqref{eq-4-1}, when $n$ is sufficiently large, $||q_{n-1}\alpha||\gg ||q_n\alpha||$.

Case 1. $||\frac{1}{2}-s_1\alpha||\ge \frac{1}{2}||q_{n-1}\alpha||$.

{\bf Claim 1.} For any $\ell \alpha\in Q_n$ and $0\le \theta< ||\frac{1}{2}-s_1\alpha||-||q_n\alpha||$, $\omega(\ell\alpha+\theta, q_n)=\omega(q_n)=1$.

\begin{proof}[Proof of the Claim 1.]
Notice that
  $$\ell\alpha+\theta+Q_n=\big(\theta+\ell\alpha+\{\alpha,2\alpha,\cdots, (q_n-\ell)\alpha\}\big)\bigsqcup\big(\theta+||q_n\alpha||+\{\alpha,2\alpha,\cdots,\ell\alpha\}\big).$$
Compare elements in $\ell\alpha+\theta+Q_n$ with elements in $Q_n$, we see that the first part of the right-hand side of the above equality is a translation
by $\theta$ of a subset of $Q_n$
and the second part is a translation by $\theta+||q_n\alpha||$ of a subset of $Q_n$. Since the distance of $s_1\alpha$ and $\frac{1}{2}$ is larger than $\theta+||q_n\alpha||$,
we have that $\omega(\ell\alpha+\theta, q_n)=\omega(q_n)=1$.
\end{proof}

For any $\ell \alpha\in Q_n$, $0\le \theta < \frac{1}{4}||q_{n-1}\alpha||$ and $0\le i\le \lfloor\frac{||q_{n-1}\alpha||}{4||q_n\alpha||}\rfloor-1$, it is easy to see that $\theta+i||q_n\alpha||+||q_n\alpha||<||\frac{1}{2}-s_1\alpha||$.
Hence for $1\le i\le \lfloor\frac{||q_{n-1}\alpha||}{4||q_n\alpha||}\rfloor-1$, by the Claim,
\begin{align*}
\omega(\ell\alpha+\theta, iq_n)&=\omega(\ell\alpha+\theta,q_n)+\omega(\ell\alpha+\theta+||q_n\alpha||,q_n)\\
&\qquad \qquad \qquad \ \ \ +\cdots+\omega(\ell\alpha+\theta+(i-1)||q_n\alpha||,q_n)\\
&=i.
\end{align*}

By \eqref{3}, $$||q_{n-1}\alpha||=q_{n-1}(-1)^{n-1}(\alpha-\frac{p_{n-1}}{q_{n-1}})\ge \frac{1}{q_n+q_{n-1}}> \frac{1}{2q_n}.$$

Let $n$ be large enough to make $\lfloor\frac{||q_{n-1}\alpha||}{6||q_n\alpha||}\rfloor\le \lfloor\frac{||q_{n-1}\alpha||}{4||q_n\alpha||}\rfloor-1$.
Since
$$\{\theta\in \mathbf{T}: \omega(\theta, iq_n)=i, 1\le i \le \lfloor\frac{||q_{n-1}\alpha||}{6||q_n\alpha||}\rfloor\}\supset \bigsqcup_{\ell=1}^{q_n}\big(\ell\alpha+[0,\frac{1}{4}||q_{n-1}\alpha||)\big),$$
we have that
\begin{align*}
  {\bf m}\{\theta\in \mathbf{T}: \omega(\theta, iq_n)=i,
1\le i \le \lfloor\frac{||q_{n-1}\alpha||}{6||q_n\alpha||}\rfloor\}\ge q_n\frac{1}{4}||q_{n-1}\alpha||> \frac{1}{8}.
\end{align*}

Case 2. $||\frac{1}{2}-s_1\alpha||< \frac{1}{2}||q_{n-1}\alpha||$.

{\bf Claim 2.} For any $\ell \alpha\in Q_n$ and $||\frac{1}{2}-s_1\alpha||< \theta\le ||q_{n-1}\alpha||-||q_n\alpha||$, $\omega(\ell\alpha+\theta, q_n)=-1$.

\begin{proof}[Proof of the Claim 2.]
As in the proof of Claim 1,
  $$\ell\alpha+\theta+Q_n=\big(\theta+\ell\alpha+\{\alpha,2\alpha,\cdots, (q_n-\ell)\alpha\}\big)\bigsqcup\big(\theta+||q_n\alpha||+\{\alpha,2\alpha,\cdots,\ell\alpha\}\big).$$
Then elements in $\ell\alpha+\theta+Q_n$ are obtained by translating elements in $Q_n$ by $\theta$ or $\theta+||q_n\alpha||$. We see that only one element $s_1\alpha$ is moved from
upper semi-circle to lower semi-circle. Hence $\omega(\ell\alpha+\theta, q_n)=-1$.
\end{proof}
Similar to Case 1, when $n$ is large enough, we also have
\begin{align*}
  {\bf m}\{\theta\in \mathbf{T}: \omega(\theta, iq_n)=i,
1\le i \le \lfloor\frac{||q_{n-1}\alpha||}{6||q_n\alpha||}\rfloor\}\ge q_n\frac{1}{4}||q_{n-1}\alpha||> \frac{1}{8}.
\end{align*}
\if Similar to the discussion in Case 1, when $n$ is large enough,
\begin{align*}
  &\{\theta\in \mathbf{T}: \omega(\theta, iq_n)=-i, 1\le i \le \lfloor\frac{||q_{n-1}\alpha||}{6||q_n\alpha||}\rfloor\}\\
  \supset & \bigsqcup_{\ell=1}^{q_n}\big(\ell\alpha+||\frac{1}{2}-s_1\alpha||+[0,\frac{1}{4}||q_{n-1}\alpha||)\big).
\end{align*}
Since
\begin{align*}
&\{\theta\in \mathbf{T}: \omega(\theta, iq_n)=i, 1\le i \le \lfloor\frac{||q_{n-1}\alpha||}{6||q_n\alpha||}\rfloor\}\\
  =&\{\theta\in \mathbf{T}: \omega(\theta, iq_n)=-i, 1\le i \le \lfloor\frac{||q_{n-1}\alpha||}{6||q_n\alpha||}\rfloor\}+\frac{1}{2},
\end{align*}
we also have
\begin{align*}
  {\bf m}\{\theta\in \mathbf{T}: \omega(\theta, iq_n)=i,
1\le i \le \lfloor\frac{||q_{n-1}\alpha||}{6||q_n\alpha||}\rfloor\}\ge q_n\frac{1}{4}||q_{n-1}\alpha||\ge \frac{1}{8}.
\end{align*}
\fi
\end{proof}

\begin{theorem}\label{ted}
Let $\mathfrak{X}=\mathfrak{X}_{\alpha,\mathfrak{Y}}$ be the DWRS with rotation $\alpha$ along $\mathfrak{Y}$, where $\alpha$ is defined
by \eqref{def-alpha} and $\mathfrak{Y}$ is a TDS with positive entropy.
Then the topological entropy dimension of $\mathfrak{X}$ is $\tau$.
\end{theorem}
\begin{proof}
Suppose $\mathfrak{Y}=(Y,S)$ and $\mathfrak{X}=(Z\times Y, T)$ as defined in section 2.
Let $\mathcal{U}$ be a finite open cover of $Z$ and $\mathcal{V}$ be a
finite open cover of $Y$. For $q_n\le k\le q_{n+1}$ and
$i=1,\cdots,k$, we now partition $Z$ due to the range of $\omega(z,i)$ for $z\in Z$. Note that $m(i)\le\omega(z,i)\le M(i)$.
Let $\widetilde U_{i,\ell}=\{z\in Z: \omega(z,i)=\ell\}$,
where $m(i)\le\ell\le M(i)$. Then $\widetilde{\mathcal U_i}=\{\widetilde U_{i,m(i)}, \widetilde U_{i,m(i)+1},
\cdots,\widetilde U_{i,M(i)}\}$ is an open cover of $Z$. Note that $z\in \sigma^{i+1}U_{i,\ell}$ if and only if $\omega(\sigma^{-i}z,0,i-1)=\omega(\sigma^{-(i+1)}z,i)=\ell$.
Hence
\begin{align*}
  \bigvee_{i=1}^kT^{-i}(\mathcal U\times\mathcal V)&\prec\bigvee_{i=1}^kT^{-i}\big((\mathcal U\vee \sigma^{i+1}\widetilde{\mathcal U_i})\times \mathcal V\big)\\
&\prec\big(\bigvee_{i=1}^k\sigma^{-i}(\mathcal U\vee
\sigma^{i+1}\widetilde{\mathcal U_i})\big)\times \big(\bigvee_{j=m(k)}^{M(k)}S^{-j}\mathcal
V\big).
\end{align*}

So
\begin{align*}
&\log N\big(\bigvee_{i=1}^kT^{-i}(\mathcal U\times \mathcal V)\big) \\
\le& \log N\big((\bigvee_{i=1}^k\sigma^{-i}(\mathcal U\vee \sigma^{i+1}\widetilde{\mathcal
U_i}))\times (\bigvee_{j=m(k)}^{M(k)}S^{-j}\mathcal
V)\big) \\
\le& \log N\big(\bigvee_{i=1}^k\sigma^{-i}\mathcal U\big)+\log N\big(\bigvee_{i=1}^k\sigma(\widetilde{\mathcal U_i})\big)+\log
N\big(\bigvee_{j=m(k)}^{M(k)}S^{-j}\mathcal V\big)\\
\le&\log N\big(\bigvee_{i=1}^k\sigma^{-i}\mathcal U\big)+\log N\big(\bigvee_{i=1}^k\sigma(\widetilde{\mathcal U_i})\big)+\big(M(k)-m(k)+1\big)\log
N(\mathcal V).
\end{align*}

For the Sturmian system $(Z,\sigma)$, its sequence entropy and entropy dimension are both zero. Noticing that $\bigvee_{i=1}^k\sigma(\widetilde{\mathcal
U_i})\prec\{\text{the cover formed by the blocks of length }k+1\}$, we have that for any $\epsilon>0$,
\begin{equation*}
 \limsup_{k\rightarrow\infty}\frac{1}{k^{\tau+\epsilon}}\bigg(\log N\big(\bigvee_{i=1}^k\sigma^{-i}\mathcal U\big)+\log N\big(\bigvee_{i=1}^k\sigma(\widetilde{\mathcal U_i})\big)\bigg)=0.
\end{equation*}
For each $k$, there exists a unique $n(k)$ such that $q_{n(k)}<k\le q_{n(k)+1}$. By Proposition \ref{lemma-Mm2},
\begin{align*}
\limsup_{k\rightarrow\infty}\frac{1}{k^{\tau+\epsilon}}\big(M(k)-m(k)+1\big)
\le& \limsup_{k\rightarrow\infty}\frac{1}{k^{\tau+\epsilon}}(6\lceil\frac{k}{q_{n(k)}}\rceil+6\sum_{j=2}^n(k)\frac{q_j}{q_{j-1}}+2q_1+5)\\
\le& \limsup_{k\rightarrow\infty}\frac{6\frac{k}{q_{n(k)}}}{k^{\tau+\epsilon}}+
\limsup_{k\rightarrow\infty}\frac{6n(k)\frac{q_{n(k)}}{q_{n(k)-1}}}{k^{\tau+\epsilon}}\\
\le& 6\big(\limsup_{k\rightarrow\infty}\frac{q_{n(k)+1}^{1-\tau}}{q_{n(k)}k^{\epsilon}}+
\limsup_{k\rightarrow\infty}\frac{q_{n(k)}^{1-\tau}n(k)}{q_{n(k)-1}k^{\epsilon}}\big)=0.
\end{align*}
Hence by (1) of Remark \ref{remark-1}, $\overline{D}(X,T)\le\tau$.

If $\tau=0$, $\overline{D}(X,T)=0$ has already been proved. In the following, we will show $\overline{D}(X,T)\ge\tau$ for $\tau>0$.

Since $h_{top}(Y,S)>0$, there exist two non-empty disjoint closed sets
$A_1,A_2$ of $Y$, such that $h_{top}(S,\{A_1^c,A_2^c\})>0$ (c.f. \cite{B}).
Let $\tilde {\mathcal{U}}= \{Z\times A_1^c,Z\times A_2^c\}$, which is an
open cover of $X$.

By Lemma \ref{lem-recurrence},
when $n$ is large enough, we can take any small closed interval from $\{\theta\in \mathbf{T}: \omega(\theta, iq_n)=i,
1\le i \le \lfloor\frac{||q_{n-1}\alpha||}{6||q_n\alpha||}\rfloor\}$ and denote it by $I$. Together with \eqref{eq-4-3},
for any fixed $1\le i\le l_n$ and any point $z\in\varphi^{-1}(I)$, $\omega(z,iq_n)=i$. By \eqref{6}, for $j=1,2$,
\begin{align*}
  T^{-iq_n}\big(\varphi^{-1}((iq_n+1)\alpha+I)\times A_j\big)&=\varphi^{-1}(\alpha +I)\times S^{-i}A_j.
\end{align*}

Hence for any $1\le \ell\le l_n$,
\begin{align*}
&N\Big(\bigvee_{i=1}^\ell T^{-iq_n}\big\{Z\times A_1^c,Z\times A_2^c\big\}\Big ) \\
\ge& N\Big(\bigvee_{i=1}^\ell T^{-iq_n}\big\{\big(\varphi^{-1}((iq_n+1)\alpha+I)\times A_1\big)^c,\big(\varphi^{-1}((iq_n+1)\alpha+I)\times A_2\big)^c\big\}\Big) \\
=& N\Big(\bigvee_{i=1}^\ell \big\{\big(\varphi^{-1}(\alpha +I)\times S^{-i}A_1\big)^c,\big(\varphi^{-1}(\alpha +I)\times S^{-i}A_2\big)^c\big\}\Big) \\
=& N\Big(\bigvee_{i=1}^\ell \big\{\varphi^{-1}(\alpha +I)^c\times Y \cup \varphi^{-1}(\alpha +I)\times
S^{-i}A_1^c, \\
&\qquad \qquad \varphi^{-1}(\alpha +I)^c\times Y \cup \varphi^{-1}(\alpha +I)\times
S^{-i}A_2^c\big\}\Big) \\
=&N\Big(\bigvee_{i=1}^\ell S^{-i}\{A_1^c,A_2^c\}\Big).
\end{align*}
Now let
$F=\{s_1<s_2<\cdots\}$ as defined in \eqref{eq-4-2}. For any $k$, there exists $n(k)$ such that $q_{n(k)}\le
s_k<q_{n(k)+1}$. Assume that $s_k=\ell q_{n(k)}$. Then $k=\sum_{i=n_0}^{n(k)-1}l_i+\ell$ and $\max\{\sum_{i=n_0}^{n(k)-1}l_i, \ell\}\ge \frac{1}{2}k$.

Hence along the sequence $F$,
\begin{align*}
&\liminf_{k\rightarrow\infty}\frac{1}{k}N\Big(\bigvee_{i=1}^kT^{-s_i}\{Z\times A_1^c,Z\times A_2^c\}\Big) \\
\ge&
\liminf_{k\rightarrow\infty}\frac{1}{k}\max \Big \{N\Big(\bigvee_{i=1}^{l_{n(k)-1}}T^{-iq_{n(k)-1}}\{Z\times A_1^c,Z\times A_2^c\}\Big), \\
&\qquad\qquad\qquad\qquad\qquad N\Big(\bigvee_{i=1}^{\ell}T^{-iq_{n(k)}}\{Z\times
A_1^c,Z\times A_2^c\}\Big)\Big\}  \\
\ge&
\liminf_{k\rightarrow\infty}\frac{1}{k}\max\Big\{N\Big(\bigvee_{i=1}^{l_{n(k)-1}}S^{-i}\{A_1^c,A_2^c\}\Big), N\Big(\bigvee_{i=1}^{\ell}S^{-i}\{A_1^c,A_2^c\}\Big)\Big\} \\
=&
\liminf_{k\rightarrow\infty}\max\Big\{\frac{l_{n(k)-1}}{k}\frac{N\Big(\bigvee_{i=1}^{l_{n(k)-1}}S^{-i}\{
A_1^c,A_2^c\}\Big)}{l_{n(k)-1}},\frac{\ell}{k}\frac{N\Big(\bigvee_{i=1}^{\ell}S^{-i}\{
A_1^c,A_2^c\}\Big)}{\ell}\Big\} \\
\ge& \frac{1}{4}h_{top}(S,\{A_1^c,A_2^c\})>0.
\end{align*}

So $F$ is an entropy generating sequence of $\tilde {\mathcal{U}}$.
Noticing that $\overline{D}(F)=\tau$, we can deduce that $\overline{D}(X,T)\ge\tau$.
Hence $\overline{D}(X,T)=\tau$.
\end{proof}

\section{Metric entropy dimension}
Let $\mathfrak{X}=\mathfrak{X}_{\alpha,\mathfrak{Y}}$ be the DWRS with rotation $\alpha$ along $\mathfrak{Y}$, where $\alpha\in E_{\tau}$ is defined as in section 4 and $\mathfrak{Y}$ is a MDS with positive metric entropy. In this section, we will show that the metric entropy dimension of $\mathfrak{X}$ also equals $\tau$.

By Remark \ref{conjugate}, in this section we assume $\mathfrak{X}=(X,\mathcal{B},\mu,T)=(\mathbf{T}\times Y,\mathcal{B}_{\mathbf{T}}\otimes \mathcal{D},{\bf m}\times \nu, T)$.

Let us consider a measurable partition of $X$, say $\mathcal{P}=\{\mathbf{T}\times A,\mathbf{T}\times A^c\}$, where $\{A,A^c\}$ is a measurable partition of $Y$ with
$h_{\nu}(S,\{A,A^c\})>0$.

\begin{proposition}\label{med}
Let $F=\{s_1<s_2<\cdots\}$ be defined by \eqref{eq-4-2}. Then $F$ is an entropy generating sequence of $\mathcal{P}$. Hence by Lemma \ref{lemma-4-1},
$$\overline{D}_{\mu}(T,\mathcal{P})\ge\tau.$$
\end{proposition}
\begin{proof}
For $\mathcal{A}=\{A_1,A_2,\cdots, A_k\}$, a collection of measurable subsets of $X$ (need not to be a partition),
we still denote $H_{\mu}(\mathcal{A})=\sum_{i=1}^k-\mu(A_i)\log\mu(A_i)$.

Let $W=A$ or $A^c$ and denote by $I_i=\{\theta\in\mathbf{T}: \omega(\theta,iq_n)=i\}$ and $I=\bigcap_{i=1}^\ell I_i$, then by \eqref{6},
\begin{align*}
  T^{-iq_n}(\mathbf{T}\times W)&=\bigcup_{\theta\in \mathbf{T}}(\theta-iq_n\alpha,S^{-\omega(\theta-iq_n\alpha,0,iq_n-1)}W)\\
  &\supset \{\theta\in\mathbf{T}: \omega(\theta,0,iq_n-1)=i\}\times S^{-i}W\\
  &=\big(I_i+\alpha\big)\times S^{-i}W.
\end{align*}
So
\begin{align*}
H_{\mu}(\bigvee_{i=1}^\ell T^{-iq_n}\mathcal{P})&\ge H_{\mu}\big(\bigvee_{i=1}^\ell \big(I_i+\alpha\big)\times S^{-i}\{A,A^c\}\big)\\
&=H_{\mu}\big(\bigcap_{i=1}^\ell\big(I_i+\alpha\big)\times \bigvee_{i=1}^\ell S^{-i}\{A,A^c\}\big)\\
&=-{\bf m}(I)\log{\bf m}(I)+{\bf m}(I)H_{\nu}(\bigvee_{i=1}^\ell S^{-i}\{A,A^c\})\\
&\ge \frac{1}{8}H_{\nu}(\bigvee_{i=1}^\ell S^{-i}\{A,A^c\}).
\end{align*}

For any $k$, there exists $n(k)$ such that $q_{n(k)}\le
s_k<q_{n(k)+1}$. Assume that $s_k=\ell q_{n(k)}$. Then $k=\sum_{i=1}^{n(k)-1}l_i+\ell$ and $\max\{\sum_{i=n_0}^{n(k)-1}l_i, \ell\}\ge \frac{1}{2}k$.

Hence along the sequence $F$,
\begin{align*}
&\liminf_{k\rightarrow\infty}\frac{1}{k}H_{\mu}\Big(\bigvee_{i=1}^kT^{-s_i}\mathcal{P}\Big) \\
\ge&
\liminf_{k\rightarrow\infty}\frac{1}{k}\max \Big \{ H_{\mu}\Big(\bigvee_{i=1}^{l_{n(k)-1}}T^{-iq_{n(k)-1}}\mathcal{P}\Big),  H_{\mu}\Big(\bigvee_{i=1}^{\ell}T^{-iq_{n(k)}}\mathcal{P}\Big)\Big\}  \\
\ge&
\liminf_{k\rightarrow\infty}\frac{1}{k}\max\Big\{\frac{1}{8}H_{\nu}(\bigvee_{i=1}^{l_{n(k)-1}} S^{-i}\{A,A^c\}), \frac{1}{8}H_{\nu}(\bigvee_{i=1}^\ell S^{-i}\{A,A^c\})\Big\} \\
=&
\liminf_{k\rightarrow\infty}\max\Big\{\frac{l_{n(k)-1}}{8k}\frac{H_{\nu}(\bigvee_{i=1}^{l_{n(k)-1}} S^{-i}\{A,A^c\})}{l_{n(k)-1}},
\frac{\ell}{8k}\frac{H_{\nu}(\bigvee_{i=1}^\ell S^{-i}\{A,A^c\})}{\ell}\Big\} \\
\ge& \frac{1}{32}h_{\nu}(S,\{A,A^c\})>0.
\end{align*}

So $F$ is an entropy generating sequence of $\mathcal{P}$.
\end{proof}

Since the metric entropy dimension is bounded from above by the topological entropy dimension, together with Theorem \ref{ted} and Proposition \ref{med}, we have
\begin{theorem}\label{mmed}
  The metric entropy dimension of $\mathfrak{X}$ equals $\tau$.
\end{theorem}

\begin{remark}
\begin{enumerate}
  \item
  Take an increasing sequence $\{\tau_n\}$ such that $\lim_{n\rightarrow\infty}\tau_n=1$. In the construction of the irrational number $\alpha=\alpha(t_1,t_2,\cdots,t_k)= [0;a_1,a_2,\cdots]$ (see \eqref{def-alpha}), when $n>k+2$, modify $a_{n+1}$ by $a_{n+1}=2\lfloor q_n^{\frac{\tau_n}{1-\tau_n}}\rfloor.$
By similar proofs of Lemma \ref{lemma-4-1} and Proposition \ref{med}, we can show that the metric entropy dimension of $\mathfrak{X}$ equals $1$ for such $\alpha$'s. Hence the topological entropy dimension is also $1$.
\item For $\tau\in [0,1)$, since $E_{\tau}$ is dense in $[0,1)$, by Theorem \ref{ted} and Theorem \ref{mmed}, the collection of irrational $\alpha$'s such that the system $\mathfrak{X}$ has entropy dimension $\tau$ is dense.
\end{enumerate}
\end{remark}

{\bf Acknowledgements}
We are grateful to the referee for the careful reading and many valuable comments to improve the paper. The first author is supported by NNSF of China (Grant No. 11790274, 11401220 and 11431012). The second author is supported in part by NRF 2010-0020946.

\end{document}